\documentclass[11pt]{article}
\usepackage[a4paper, total={7in, 8in}]{geometry}
\usepackage{enumerate}
\usepackage{amsmath}
\usepackage{amssymb,latexsym}
\usepackage{amsthm}
\usepackage{color}
\usepackage{cancel}
\usepackage{graphicx}
\usepackage{cite}
\usepackage{longtable}
\usepackage{amscd}
\usepackage{lineno}

\makeatletter
\usepackage{comment}
\let\wfs@comment@comment\comment
\let\comment\@undefined

\usepackage{changes}
\let\wfs@changes@comment\comment
\let\comment\@undefined

\newcommand\comment{%
    \ifthenelse{\equal{\@currenvir}{comment}}
    {\wfs@comment@comment}
    {\wfs@changes@comment}%
}

\definechangesauthor[name=Daniele, color=red]{DAN}

\newtheorem{thm}{Theorem}[section]

\newtheorem{prop}[thm]{Proposition}
\newtheorem{cor}[thm]{Corollary}

\newtheorem*{Main}{Main Theorem}

\newtheorem{rem}[thm]{Remark}

\title{On Permutation Trinomials of the type $X^{q^2-q+1}+AX^{q^2}+BX$ over $\mathbb{F}_{q^3}$}
\author{Daniele Bartoli\thanks{Dipartimento di Matematica e Informatica, Universit\`a degli Studi di Perugia,  Perugia, Italy. daniele.bartoli@unipg.it}, and Francesco Ghiandoni\thanks{Dipartimento di Matematica e Informatica ``Ulisse Dini", Universit\`a  degli studi di Firenze, Firenze, Italy, francesco.ghiandoni@unifi.it}}
\begin{document}

\maketitle
\begin{abstract} 
Necessary and sufficient conditions on $A,B\in \mathbb{F}_{q^3}^*$ for $f(X)=X^{q^2-q+1}+AX^{q^2}+BX$ being a permutation polynomial of $\mathbb{F}_{q^3}$ are investigated via a connection with algebraic varieties over finite fields. 
\end{abstract}


\section{Introduction}\label{sec:intro}
Let $q$ be a prime power and denote by $\mathbb{F}_q$ the finite field with $q$ elements. A polynomial $f(x)\in  \mathbb{F}_q[x]$ is a \emph{permutation polynomial} (PP) of $\mathbb{F}_q$ if it is a bijection of  $\mathbb{F}_q$ into itself. Permutation polynomials were first studied by Hermite and Dickson; see \cite{MR1502214,hermit}.

 
In general,  simple structures or additional extraordinary properties are  required by applications of PPs in other areas of mathematics and engineering, such as cryptography, coding theory, or combinatorial designs. In this case, permutation polynomials meeting these criteria are usually difficult to find. For a deeper introduction to the connections of PPs with other fields of mathematics we refer to \cite{MR1405990,MR3293406,MR3087321} and the references therein. 

Permutation polynomials of monomial or binomial type have been widely investigated in the last decades. Although much less is known for polynomials having more than 2 terms, specific families of trinomials and quadrinomials were considered; see for instance  \cite{MR4014383, MR4214973, MR4186036, MR4104437, MR3925610, https://doi.org/10.48550/arxiv.2203.04216, MR4488068, hermit, MR4183827, MR4033170, MR4109896, MR3293406, MR4192803, MR3895827, MR4296215, MR4550904, MR3835238, MR2495253, zieve, MR4154425, MR4142308, MR4079332, MR4048039, MR4014382, MR3990981, MR3975902, MR3959342, MR3928513, MR3925616, MR3874713, MR3847017, MR3834890, MR3813096, MR3807840, MR3798783, MR3764685, MR3759797, MR3759789, MR3724861, MR3682991, MR3681077, MR3655747, MR3631351, MR3592813, MR3575462, MR3528704, MR3376621, MR3368797, MR3329981, MR3299134, MR3283622, MR3162811},  where in most of the cases $q$ is a square. \\
When investigating the permutational properties of a polynomial, a fruitful connection with algebraic curves is provided by the following observation (we refer to \cite{MR1042981} for a comprehensive introduction on algebraic curves). Given a polynomial $f(x)\in \mathbb{F}_q[x]$, let us consider the curve $\mathcal{C}_f$ with affine equation 
\begin{equation}\label{Eq:C_f}
\mathcal{C}_f : \frac{f(X)-f(Y)}{X-Y}=0.
\end{equation}
 A standard approach to the problem of deciding whether $f(x)$ is a PP is the investigation of the set of $\mathbb{F}_q$-rational points of $\mathcal{C}_f$. In fact, if $a\neq b$ are two distinct elements of $\mathbb{F}_q$ such that $f(a)=f(b)$, then the $\mathbb{F}_q$-rational point $(a,b)$ belongs to $\mathcal{C}_f$ and does not lie on the line $X-Y=0$. On the other hand, if $(a,b)$ is an $\mathbb{F}_q$-rational point of $\mathcal{C}_f$ not lying on $X-Y=0$, then $f(a)=f(b)$ and so $f(x)$ is not a PP. 
 
In general to decide whether an algebraic curve $\mathcal{C}_f$  defined over $\mathbb{F}_q$ possesses suitable $\mathbb{F}_q$ rational points is a hard task. A standard way to bypass such a problem is to prove the existence of absolutely irreducible components defined over $\mathbb{F}_q$ of $\mathcal{C}$ and then apply to such a component the celebrated Hasse-Weil Theorem (or refinements of it); see \cite{MR4109896}.

Unfortunately, this machinery fails whenever the degree of $\mathcal{C}$ is too large with respect to $q$, namely larger than roughly $\sqrt[4]{q}$. 

To overcome such a problem, for specific families of polynomials of large degree, one can investigate higher dimensional varieties attached to the starting polynomial ${f}$.

In this paper we consider polynomials of the form 
\begin{equation}
    f_{A,B}(X)=X^{q^2-q+1}+AX^{q^2}+BX \in \mathbb{F}_{q^3}[X],
\end{equation}
where $A,B$ are nonzero elements of $\mathbb{F}_{q^3}$ and $q$ is even.
Such a family has been partially investigated in \cite{MR4296215}, where the authors proved the following result.
\begin{thm} \label{thm lang xu ecc}
If $A^{q^2+q+1}=B^{1+q+q^2}$ and $A^qB \in \mathbb{F}_q \setminus \{0,1\},$ then $f_{A,B}(X)=X^{q^2-q+1}+AX^{q^2}+BX$ is a PP  over $\mathbb{F}_{q^3}.$
\end{thm}

Exploiting a connection with algebraic surfaces (see Section \ref{Section:Connection}) and generalizations of Hasse-Weil Theorem to higher dimensions, we are able to prove that, for $q$ large enough and apart from a few exceptions, the conditions on parameters $A,B$ expressed in Theorem \ref{thm lang xu ecc} are also sufficient.
\begin{Main}
Let $q=2^m,$ $m \geq 19.$ Then $f_{A,B}=X^{q^2-q+1}+AX^{q^2}+BX \in \mathbb{F}_{q^3}[X]$ is a PP over $\mathbb{F}_{q^3}[X]$ if and only if one of the following conditions holds:
\begin{itemize}
    \item $A^{1+q+q^2}=B^{1+q+q^2}$ and $A^qB \in \mathbb{F}_q \setminus \{0,1\};$
    \item $A^qB=1$ and $B^{1+q+q^2}\neq 1.$
\end{itemize}
\end{Main}

\section{Preliminaries}
Let $q=2^m,$ $m \in \mathbb{N},$ and denote by $\mathbb{F}_q$ the finite field with $q$ elements.
As a notation, $\mathbb{P}^r(\mathbb{F}_q)$ and $\mathbb{A}^{r}(\mathbb{F}_{q})$ (or $\mathbb{F}_{q}^r$) denote the projective and the affine space of dimension $r\in \mathbb{N}$ over the finite field $\mathbb{F}_{q}$ respectively. A variety, and more specifically a curve or a surface (i.e. a variety of dimension 1 or 2 respectively),
 is described by a certain set of equations with coefficients in  $\mathbb{F}_{q}.$ We say that a variety $\mathcal{V}$ is \textit{absolutely irreducible} if there are no varieties $\mathcal{V}'$
 and $\mathcal{V}''$ defined over the algebraic closure of $\mathbb{F}_q$ (denoted by $\overline{\mathbb{F}_q}$) and different
from $\mathcal{V}$ such that $\mathcal{V}=\mathcal{V}' \cup \mathcal{V}''$. If a variety $\mathcal{V} \subseteq  \mathbb{P}^r(\mathbb{F}_q)$ is defined by
$F_i(X_0, \dots , X_r) = 0,$ for $i = 1, \dots  s,$ an $\mathbb{F}_q$-rational point of $\mathcal{V}$ is a point
$(x_0 : \dots : x_r) \in \mathbb{P}^r(\mathbb{F}_q)$ such that $F_i(x_0, \dots, x_r) = 0,$ for $i = 1, \dots , s.$  The set of the $\mathbb{F}_q$-rational points of $\mathcal{V}$ is usually
denoted by $\mathcal{V}(\mathbb{F}_q).$ If $s=1,$ $\mathcal{V}$ is called a hypersurface and it is absolutely irreducible if the corresponding polynomial $F(X_1,\dots,X_r)$ is absolutely irreducible, i.e. it possesses no non-trivial factors over $\overline{\mathbb{F}_q}.$
Moreover, we say that $\mathcal{V}$ is a variety of degree $d$ (and write $\deg(\mathcal{V})=d$) if $d=\#(\mathcal{V}\cap H),$ where $H \subseteq\mathbb{P}^r(\overline{\mathbb{F}_q}) $ is a general projective subspace of dimension $r-s.$  To determine the degree of a variety is generally not straighforward; however an upper bound to $\deg(\mathcal{V)}$ is given by $\prod_{i=1}^{s}\deg(F_i).$ 
We also recall that the Frobenius map $\Phi_q: x \mapsto x^q$ is an automorphism of $\mathbb{F}_{q^3}$ and generates the group $Gal(\mathbb{F}_{q^3} / \mathbb{F}_q)$ of automorphisms of  $\mathbb{F}_{q^3}$ that fixes $\mathbb{F}_{q}$ pointwise.
The Frobenius automorphism induces also a collineation of $\mathbb{P}^r(\mathbb{F}_q^3)$ and an automorphism of $\mathbb{F}_{q^3}[X_0,\dots,X_r].$ \\
A crucial point in our investigation of permutation trinomials over $\mathbb{F}_{q^3}$ is to prove the existence of suitable $\mathbb{F}_q$-rational points in algebraic surfaces $\mathcal{V}$ attached to each permutation trinomial. This is reached by proving the existence of absolutely irreducible $\mathbb{F}_q$-rational components in $\mathcal{V}$ and lower bounding the number of their $\mathbb{F}_q$-rational points. To this end, generalizations of Lang-Weil type bounds
for algebraic varieties are needed.
To ensure the existence of a suitable $\mathbb{F}_q$-rational point of $\mathcal{V},$ we report the following
result.
\begin{thm}\cite[Theorem 7.1]{MR2206396} \label{thm lang weil versione tredici terzi}
    Let $V \subseteq \mathbb{A}^n(\mathbb{F}_q)$ be an absolutely irreducible variety
defined over $\mathbb{F}_q$ of dimension $r  > 0$ and degree $\delta.$ If $q > 2(r + 1)\delta^2,$ then the following
estimate holds:
\begin{equation}
  \big|\#(\mathcal{V}( \mathbb{A}^n(\mathbb{F}_q))) - q^r
\big| \leq (\delta - 1)(\delta - 2)q^{r-1/2} + 5\delta^{13/3}q^{r-1}.  
\end{equation}
\end{thm}

\section{Connection with algebraic surfaces}\label{Section:Connection}
\begin{comment}
 Let $C_{f_{A,B}}$ be the curve of affine equation $(f_{A,B} (X) - f_{A,B} (Y))/(X - Y) = 0.$ It is well-known 
that $C_{f_{A,B}}$ is defined over $\mathbb{F}_{q^3},$ and that $f_{A,B}$ is a bijection over $\mathbb{F}_{q^3}$ if and only if $C_{f_{A,B}}$ has no
$\mathbb{F}_{q^3}$-rational points off the line $X=Y.$
However, in order to obtain a curve defined by a polynomial with a cubic dependence on $X,X^q,X^{q^2},Y,Y^q,Y^{q^2},$ we consider the curve $\mathcal{D}_{f_{A,B}} = \mathcal{C}_{f_{A,B}} \cup \{ X^qY^q=0\}. $   
\end{comment}
The following proposition gives the explicit link between permutation polynomial and algebraic curves.
\begin{prop} \label{prop collegam curve}
If $f_{A,B}(X)=X^{q^2-q+1}+AX^{q^2}+BX \in \mathbb{F}_{q^3}[X]$ is a PP on $\mathbb{F}_{q^3}$ then the curve $$ \mathcal{C}_{f_{A,B}}:\varphi_{f_{A,B}}:= \dfrac{Y^q(X^{q^2+1}+AX^{q^2+q}+BX^{q+1})+X^q(Y^{q^2+1}+AY^{q^2+q}+BY^{q+1})}{X^qY^q(X+Y)}=0 $$ has no affine $\mathbb{F}_{q^3}$-rational points off $XY(X+Y)=0.$
Moreover, if $f(X)$ has no trivial roots in $\mathbb{F}_{q^3},$ then  the converse is also true.
\end{prop}
\begin{proof}
 If $f_{A,B}$ is a PP on $\mathbb{F}_{q^3},$  then for all $x,y \in \mathbb{F}_{q^3},$ $x\neq 0 \neq y$ and $x \neq y,$ \[0 \neq x^qy^q(f_{A,B}(x)+f_{A,B}(y))=y^q(x^{q^2+1}+Ax^{q^2+q}+Bx^{q+1})+x^q(y^{q^2+1}+Ay^{q^2+q}+By^{q+1})\]
 and so $C_{f_{A,B}}$ has no affine $\mathbb{F}_{q^3}$-rational points off $XY(X+Y)=0.$
 The converse also holds.
\end{proof}
Unfortunately, $\deg(C_{f_{A,B}})=q^2-1 > \sqrt[4]{q^3}$ so Hasse-Weil type theorems do not ensure the existence of $\mathbb{F}_{q^3}$-rational points for the curve $C_{f_{A,B}}.$ \\
We overcome this problem by considering a link between $C_{f_{A,B}}$ and a suitable surface in $\mathbb{P}^5(\mathbb{F}_{q^3})$ of small degree,  in order to apply Theorem \ref{thm lang weil versione tredici terzi}. \\ 
Define $\mathcal{C}: X^qY^q(X+Y)\varphi_{f_{A,B}}(X,Y)=0.$
Let $\left\lbrace 	\xi,\xi^q,\xi^{q^2}\right\rbrace $ be a normal basis of $\mathbb{F}_{q^3}$ over $\mathbb{F}_{q}.$
Write \\$X=x_0\xi+x_1\xi^q+x_2\xi^{q^2}$ and $$Y=y_0\xi+y_1\xi^q+y_2\xi^{q^2}, \quad \textrm{where}  \quad x_i,y_i\in \mathbb{F}_q, \quad i=0,1,2.$$  Clearly, the invertible map   \begin{equation}
\Lambda: \mathbb{F}^2_{q^3} \rightarrow \mathbb{F}^6_{q}, \ \ \ \ \ \ \  (X,Y) \mapsto (x_0,x_1,x_2,y_0,y_1,y_2)
\end{equation} induces a surjective map between the set of affine $\mathbb{F}_{q^3}$-rational points of $\mathcal{C}$ and the set of projective $\mathbb{F}_{q}$-rational points of a suitable  surface $\mathcal{V} \subseteq \mathbb{P}^5(\mathbb{F}_q)$ of equation \[f_1(x_0,x_1,x_2,y_0,y_1,y_2)\xi+f_2(x_0,x_1,x_2,y_0,y_1,y_2)\xi^q+f_3(x_0,x_1,x_2,y_0,y_1,y_2)\xi^{q^2}=0,\] i.e.
\begin{equation*}
	\mathcal{V}:\left\{
	\begin{array}{l}
		f_1(x_0,x_1,x_2,y_0,y_1,y_2)=0\\ 
		f_2(x_0,x_1,x_2,y_0,y_1,y_2)=0\\
	   f_3(x_0,x_1,x_2,y_0,y_1,y_2)=0,
	\end{array}
	\right.  
\end{equation*}
where $f_i \in \mathbb{F}_q[x_0,x_1,x_2,y_0,y_1,y_2]$ and $\deg(f_i)=3,$ for $i=1,2,3.$ \\
Setting
\begin{eqnarray*}
 X&=&x_0\xi+x_1\xi^q+x_2\xi^{q^2}=:X_0   \hspace{2cm}   Y=y_0\xi+y_1\xi^q+y_2\xi^{q^2}=:Y_0 \\
 X^q&=&x_2\xi+x_0\xi^q+x_1\xi^{q^2}=:X_1  \hspace{2cm}   Y^q=y_2\xi+y_0\xi^q+y_1\xi^{q^2}=:Y_1\\
 X^{q^2}&=&x_1\xi+x_2\xi^q+x_0\xi^{q^2}=:X_2  \hspace{2cm}   Y^{q^2}=y_1\xi+y_2\xi^q+y_0\xi^{q^2}=:Y_2,
 \end{eqnarray*}
it follows that the map        $ \  \Theta : (x_0:x_1:x_2:y_0:y_1:y_2) \mapsto (X_0:X_1:X_2:Y_0:Y_1:Y_2)$ is a projectivity of $ \mathbb{P}^5(\mathbb{F}_{q^3}).$ 
Moreover the Frobenius automorphism $\Phi_q$ induces a collineation of $ \mathbb{P}^5(\mathbb{F}_{q^3}) $ that fixes $\mathcal{V},$ so the surface  $\Theta^{-1}(\mathcal{V})$ is fixed by the collineation $$\Psi_q :  \mathbb{P}^5(\mathbb{F}_{q^3}) \rightarrow \mathbb{P}^5(\mathbb{F}_{q^3}), \ \ \ \     (u_0:u_1:u_2:v_0:v_1:v_2)\mapsto (u_2^{q^2}:u_0^{q^2}:u_1^{q^2}:v_2^{q^2}:v_0^{q^2}:v_1^{q^2}).$$  We will denote by  $\Psi_q$ also the automorphism  of $\mathbb{F}_{q^3}[X_0,X_1,X_2,Y_0,Y_1,Y_2]$  given by $$H(X_0,X_1,X_2,Y_0,Y_1,Y_2) \mapsto H^q(X_1,X_2,X_0,Y_1,Y_2,Y_0),$$ where $H^q$
denotes the polynomial obtained raising the coefficients of $H$
to the power $q.$ \\ Observing that
\begin{eqnarray*}
    X^qY^q(X+Y)\varphi_{f_{A,B}}&=&Y^q(X^{q^2}X+AX^{q^2}X^q+BX^{q}X)+X^q(Y^{q^2}Y+AY^{q^2}Y^q+BY^{q}Y) \\ &=:&\overline{F}(X,X^q,X^{q^2},Y,Y^q,Y^{q^2}),
\end{eqnarray*} with  $\overline{F} \in \mathbb{F}_{q^3}[X_0,X_1,X_2,Y_0,Y_1,Y_2],$
we get that the surface $\Theta^{-1}(\mathcal{V})$ is defined by
\begin{equation*}
	\left\{
	\begin{array}{l}
		\overline{F}=0\\ 
		\Psi_q(\overline{F})=0\\
	   \Psi_q^2(\overline{F})=0,
	\end{array}
	\right.  
\end{equation*}
i.e.
\begin{equation} \label{eq varietà fissata da shift}
	\Theta^{-1}(\mathcal{V}):\left\{
	\begin{array}{l}
		Y_1(X_0X_2+AX_1X_2+BX_0X_1)=X_1(Y_0Y_2+AY_1Y_2+BY_0Y_1)\\ 
		Y_2(X_0X_1+A^qX_0X_2+B^qX_1X_2)=X_2(Y_0Y_1+A^qY_0Y_2+B^qY_1Y_2)\\
	   Y_0(X_1X_2+A^{q^2}X_0X_1+B^{q^2}X_0X_2)=X_0(Y_1Y_2+A^{q^2}Y_0Y_1+B^{q^2}Y_0Y_2).
	\end{array}
	\right.  
\end{equation}
Again the composition $$ \Theta \circ \Lambda : \mathbb{F}^2_{q^3} \rightarrow \mathbb{F}^6_{q}, \ \ \ \  (X,Y) \mapsto (X_0,X_1,X_2,Y_0,Y_1,Y_2)$$ induces a surjective map between the set of affine $\mathbb{F}_{q^3}$-rational points of $\mathcal{C}$ and the set of points of $\Theta^{-1}(\mathcal{V})$ in $\mathbb{P}^5(\mathbb{F}_{q^3})$ which are fixed by $\Psi_q .$  
Furthermore,  
one can observe that every component of $\Theta^{-1}(\mathcal{V})$ which is fixed by $\Psi_q$ corresponds to an $\mathbb{F}_q$-rational component of $\mathcal{V}$ and viceversa. \\
We are now interested in studying the existence of absolutely irreducible components of $\Theta^{-1}(\mathcal{V})$ which are fixed by $\Psi_q.$   \\
Clearly, from $\varphi_{f_{A,B}}(X,0)=\varphi_{f_{A,B}}(0,Y)=\varphi_{f_{A,B}}(X,X)=0$ one deduces that the planes 
\begin{equation}
    \mathcal{U}_1: \left\{
	\begin{array}{l}
	X_2=0\\ 
		X_1=0\\
	  X_0=0
	\end{array}
	\right. 
	 \ \ \ \ \ \ \ \ \mathcal{U}_2:\left\{
	\begin{array}{l}
	 Y_2=0\\ 
		Y_1=0 \\
	  Y_0=0
	\end{array}
	\right.
	\ \ \ \ \ \ \ \ \mathcal{U}_3:
	\left\{
	\begin{array}{l}
	 X_2=Y_2\\ 
		X_1=Y_1 \\
	  X_0=Y_0
	\end{array}
	\right.
	\end{equation}
	are absolutely irreducible components of $\Theta^{-1}(\mathcal{V}).$ It may be noticed that  points on $XY(X+Y)=0$ correspond via the map $\Theta \circ \Lambda$ to points on $\mathcal{U}_1\cup\mathcal{U}_2\cup \mathcal{U}_3$ and vice versa. \\
	It is not hard to see that also the surface
\begin{equation*}
	\mathcal{U}:\left\{
	\begin{array}{l}
		X_2=0\\ 
		X_0=0\\
	   Y_0Y_2+AY_1Y_2+BY_0Y_1=0
	\end{array}
	\right.  
\end{equation*}
is an absolutely irreducible component of $\Theta^{-1}(\mathcal{V}).$
Moreover, since $\varphi_{f_{A,B}}(X,Y)=\varphi_{f_{A,B}}(Y,X),$  $\Theta^{-1}(\mathcal{V}) $ is fixed by the projectivity $$\Sigma : (X_0,X_1,X_2,Y_0,Y_1,Y_2) \mapsto (Y_0,Y_1,Y_2,X_0,X_1,X_2) $$ over $\mathbb{P}^5(\mathbb{F}_{q^3}),$ as well as by $\Psi_q.$ Therefore $\mathcal{U},\Sigma(\mathcal{U}),\Psi_q(\mathcal{U}),\Sigma(\Psi_q(\mathcal{U})),\Psi_q^2(\mathcal{U}),\Sigma(\Psi_q^2(\mathcal{U}))$ are six distinct components of $\Theta^{-1}(\mathcal{V}) $ of degree 2, which are not fixed by $\Psi_q.$
Thus $\Theta^{-1}(\mathcal{V})$ splits in at least ten components
\begin{equation} \label{eq decomp of theta v}
     \Theta^{-1}(\mathcal{V})\supseteq \left(\mathcal{U}_1 \cup \mathcal{U}_2 \cup \mathcal{U}_3 \cup \mathcal{U}\cup \Sigma(\mathcal{U}) \cup \Psi_q(\mathcal{U}) \cup \Sigma(\Psi_q(\mathcal{U})) \cup \Psi_q^2(\mathcal{U})\cup \Sigma(\Psi_q^2(\mathcal{U})) \right) \cup \mathcal{W},
\end{equation}
where $\mathcal{W}$ is a surface fixed by $\Psi_q$ and of degree at most $27-3-12=12.$\\
From now on we will investigate the absolutely irreducibility of $\mathcal{W}.$
\begin{comment}
\begin{equation}
    \Theta^{-1}(\mathcal{V}): 
    \left\{
	\begin{array}{l}
	X_2=0\\ 
		X_1=0\\
	  X_0=0
	\end{array}
	\right. 
	\vee 
	\left\{
	\begin{array}{l}
	 Y_2=0\\ 
		Y_1=0 \\
	  Y_0=0
	\end{array}
	\right.
	\vee 
	\left\{
	\begin{array}{l}
	 X_2=Y_2\\ 
		X_1=Y_1 \\
	  X_0=Y_0
	\end{array}
	\right.
	\vee
	\bigcup_{i}
\end{equation}
\end{comment}

From the first equation of  System \eqref{eq varietà fissata da shift} we obtain
$X_2=X_1 \cdot \frac{Y_2(Y_0+AY_1)+BY_1(X_0+Y_0)}{Y_1(X_0+AX_1)}.$
\\ By replacing $X_2$ in the second equation of \eqref{eq varietà fissata da shift}, we get $X_1=0$ (from which we obtain $\mathcal{U}_1 \cup \Psi_{q^2}(\mathcal{U}$)) or $$ X_1=\dfrac{[Y_2(Y_0+AY_1)+BY_1(X_0+Y_0)][A^qY_2(X_0+Y_0)+Y_1(Y_0+B^qY_2)]+X_0^2Y_1Y_2}{Y_2[AX_0Y_1+B^q(Y_2(Y_0+AY_1)+BY_1(X_0+Y_0))]}. $$
Replacing $X_1$ and $X_2$ as rational functions in $X_0,Y_0,Y_1,Y_2$ in the last equation of the system, we obtain $X_0=Y_0$ (i.e. the component $\mathcal{U}_3$) or
\begin{comment}
\begin{equation*}
	\left\{
	\begin{array}{l}
	X_2=X_1 \cdot \frac{Y_2(Y_0+AY_1)+BY_1(X_0+Y_0)}{Y_1(X_0+AX_1)}\\ 
		X_1=\frac{[Y_2(Y_0+AY_1)+BY_1(X_0+Y_0)][A^qY_2(X_0+Y_0)+Y_1(Y_0+B^qY_2)]+X_0^2Y_1Y_2}{Y_2[AX_0Y_1+B^q(Y_2(Y_0+AY_1)+BY_1(X_0+Y_0))]}\\
	  G(X_0,Y_0,Y_1,Y_2)=0
	\end{array}
	\right.  
\end{equation*}
i.e (by replacing $X_1$ in the first equation)    
\end{comment}

\begin{equation*}
	\mathcal{W} : \left\{
	\begin{array}{l}
	X_2= \frac{[Y_2(Y_0+AY_1)+BY_1(X_0+Y_0)][A^qY_2(X_0+Y_0)+Y_1(Y_0+B^qY_2)]+X_0^2Y_1Y_2}{Y_1[(A^{q+1}+B^q)X_0Y_2+A^{q+1}Y_0Y_2+AB^qY_1Y_2+AY_0Y_1]}\\ 
		X_1=\frac{[Y_2(Y_0+AY_1)+BY_1(X_0+Y_0)][A^qY_2(X_0+Y_0)+Y_1(Y_0+B^qY_2)]+X_0^2Y_1Y_2}{Y_2[AX_0Y_1+B^q(Y_2(Y_0+AY_1)+BY_1(X_0+Y_0))]}\\
	  G(X_0,Y_0,Y_1,Y_2)=0,
	\end{array}
	\right.  
\end{equation*}
with $G \in \mathbb{F}_{q^3}[X_0,Y_0,Y_1,Y_2]$ homogeneous. 
By  MAGMA computations\cite{MR1484478}, we get $\deg(G(X_0,Y_0,Y_1,Y_2))=8$ and \begin{equation}
    G_*(X_0,Y_0,Y_1)=G(X_0,Y_0,Y_1,1)=\alpha(Y_0,Y_1)X_0^3+\beta(Y_0,Y_1)X_0^2+\gamma(Y_0,Y_1)X_0+\delta(Y_0,Y_1),
\end{equation} where $\alpha,\beta,\gamma,\delta \in \mathbb{F}_{q^3}[Y_0,Y_1].$\\
If $\alpha$ is not zero, then $\deg_{X_0}(G_*)=3$ and $G_*$ is not absolutely irreducible if and only if there exists a  factor  $\epsilon(Y_0,Y_1)X_0+\sigma(Y_0,Y_1)$ of $G_*,$ where $\epsilon(Y_0,Y_1) \mid \alpha(Y_0,Y_1)$ and $\sigma(Y_0,Y_1) \mid \delta(Y_0,Y_1).$
By  MAGMA computations, it results that
\begin{eqnarray*}
\alpha(Y_0,Y_1)&=&(A^qB+1)(A^{1+q+q^2}+AB^{q^2}+A^qB+A^{q^2}B^q+B^{1+q+q^2}+1)Y_0Y_1^2, \\
\delta(Y_0,Y_1)&=&(Y_0Y_1+B^qY_1Y_2+A^qY_0Y_2)_*^2(AY_1Y_2+BY_0Y_1+Y_0Y_2)_*^2=:M_*^2N_*^2.
\end{eqnarray*}
Thus $\epsilon(Y_0,Y_1)= Y_0^iY_1^j$ and $\sigma(Y_0,Y_1)=\lambda M_*^\ell N_*^k,$ with $0\leq i \leq 1 ,$ $0 \leq j,\ell,k \leq 2,$ for some $\lambda \in \mathbb{F}_{q^3} \! \setminus \! \{0\}.$ Also $\epsilon(Y_0,Y_1)X_0+\sigma(Y_0,Y_1) \mid G_*$ if and only if $G_*(\frac{\lambda M_*^\ell N_*^k}{Y_0^iY_1^j},Y_0,Y_1)=0$ in $\mathbb{F}_{q^3}(Y_0,Y_1),$ that is (clearing the  denominators)  \begin{equation} \label{eq alpha beta gamma delta}
 \alpha M_*^{3\ell}N_*^{3k}+\beta M_*^{2\ell}N_*^{2k}Y_0^i Y_1^j +\gamma M_*^{\ell}N_*^{k}Y_0^{2i} Y_1^{2j}+\delta Y_0^{3i} Y_1^{3j}=0   
\end{equation}   in $\mathbb{F}_{q^3}[Y_0,Y_1].$ 

We first investigate the case $\alpha=0$.
\begin{prop} \label{prop f PP se e solo se norma eq 1}
 Let  $A,B \in \mathbb{F}_{q^3}.$ Then $f_{A,B}$ is not a permutation polynomial over $\mathbb{F}_{q^3}$ in any of the following cases.
 \begin{itemize}
     \item [(i)] $A^qB+1=0$ and $A^{1+q+q^2}=1;$
     \item [(ii)] $A^qB+1 \neq 0$ and $A^{1+q+q^2}+AB^{q^2}+A^qB+A^{q^2}B^q+B^{1+q+q^2}+1=0.$
 \end{itemize}
\end{prop}
\begin{proof}
We want to find a non-trivial root of $f_{A,B}$ in $\mathbb{F}_{q^3};$ the statement will follow.
\begin{itemize}
    \item [(i)] Clearly, $f_{A,B}$ has a non-trivial root in $\mathbb{F}_{q^3}$ if and only if there exists $u \in \mathbb{F}_{q^3}, u^{1+q+q^2}=1,$ such that \begin{equation*} 
   u^q+Au^{q+1}+A^{-q}=0 
\end{equation*} (see \cite[Theorem 1]{MR4296215}), i.e. such that
\begin{equation} \label{eq caso aqb eq 1 radici}
   (Au)^q+(Au)^{q+1}+1=0. 
\end{equation}
It is well-known that all the $q+1$ solutions of the equation $y^{q+1}+y^q+1=0$ are elements of $\mathbb{F}_{q^3}$ and satisfy $y^{1+q+q^2}=1$ (see for istance \cite{MR3087321}). Thus consider $\overline{u}=A^{-1}\overline{y},$ where $\overline{y}^{q+1}+\overline{y}^q+1=0;$  such $\overline{u}$ satisfies $\overline{u}^{1+q+q^2}=1$ if and only if $A^{1+q+q^2}=1.$ From the surjectivity of the map $\gamma: \mathbb{F}_{q^3} \rightarrow \{x\in \mathbb{F}_{q^3} : x^{1+q+q^2}=1\},$ $x \mapsto x^{q-1},$  we get the statement.
\item [(ii)]  Again, $f$ has a non-trivial root in $\mathbb{F}_{q^3}$ if and only if there exists $u \in \mathbb{F}_{q^3}, u^{1+q+q^2}=1,$ such that \begin{equation} \label{eq u a b}
   u^q+Au^{q+1}+B=0;
\end{equation}  (see \cite{MR4296215}). \\
In \cite[Theorem 1]{MR4296215}, the authors proved that if $u$ is a solution of \eqref{eq u a b} with $u^{1+q+q^2}=1,$ then $(A+B^{1+q})u=A^qB+1.$
Moreover, from \[
    \left\{
	\begin{array}{l}
	A^{1+q+q^2}+AB^{q^2}+A^qB+A^{q^2}B^q+B^{1+q+q^2}+1=0\\ 
		A+B^{1+q}=0
	\end{array}
	\right.  
\] it follows that $A^{1+q+q^2}=B^{1+q+q^2}=1,$ a contradiction to $A^qB+1 \neq 0.$ So $u=\frac{A^qB+1}{A+B^{1+q}}$ is an element of $\mathbb{F}_{q^3}^*.$ We now prove that such $u$ has the required properties.
We have that $u^{1+q+q^2}=1$ if and only if $(A^qB+1)^{1+q+q^2}=(A+B^{1+q})^{1+q+q^2} $ that is $B^{1+q+q^2}(A^{1+q+q^2}+AB^{q^2}+A^qB+A^{q^2}B^q+B^{1+q+q^2}+1)=0.$ 
Finally, by replacing $u=\frac{A^qB+1}{A+B^{1+q}}$ in Equation \eqref{eq u a b} we obtain
\begin{eqnarray*}
& & \left(\frac{A^qB+1}{A+B^{1+q}}\right)^q+A \cdot \left(\frac{A^qB+1}{A+B^{1+q}}\right)^{q+1}+B \\
&=& \frac{A^{q^2}B^q+1}{A^q+B^{q+q^2}}\left[1+\frac{A(A^qB+1)}{A+B^{1+q}}\right]+B \\
&=& \frac{(A^{q^2}B^q+1)(B^{1+q}+A^{1+q}B)}{(A+B^{1+q})(A^q+B^{q+q^2})}+B \\
&=& \frac{B^{1+q}(A^{1+q+q^2}+AB^{q^2}+A^qB+A^{q^2}B^q+B^{1+q+q^2}+1)}{(A+B^{1+q})(A^q+B^{q+q^2})} \\
&=&0.
\end{eqnarray*}
\end{itemize} 
\end{proof}
The case $A^qB=1$ and $A^{1+q+q^2} \neq 1$ is investigated in Proposition \ref{CASE:pp}. 

In the rest of this section we assume $\alpha\neq 0$.  

\begin{prop} \label{prop casi iniziali da escludere}
Let $\alpha \neq 0.$ Then $f_{i,j,\ell,k}:=\lambda Y_0^iY_1^jX_0+M_*^\ell N_*^k \nmid G_*$ in any of the following cases:
\begin{itemize}
\item [(a)]$\ell + k \geq 3;$
\item [(b)]$\ell + k =2$ and $i+j < 2;$
\item [(c)]$\ell + k =0;$
\item [(d)]$\ell + k =1$ and $i+j \geq 2.$
\end{itemize}
\end{prop}
\begin{proof}
By MAGMA computations, we get   $\deg(\beta)=5,$ $\deg(\gamma)=7$ (if $\beta,\gamma$ are not zero) and $\deg(\delta)=8.$ By a direct check  the leading term of the polynomial in Equation \eqref{eq alpha beta gamma delta} appears in  $\alpha M_*^{3\ell}N_*^{3k}$ in cases $(a)$ and $(b),$ and in $\delta Y_0^{3i} Y_1^{3j}$ in cases $(c)$ and $(d)$. The claim follows.
\end{proof}
We consider now all the pending cases,
\begin{itemize}
    \item $\ell + k =2$ and $i+j = 2,3;$
    \item $\ell + k =1$ and $i+j=0,1.$
\end{itemize}
\begin{prop}   \label{prop tre fattori lineari}
Let $\alpha \neq 0.$ If $G_*$ is reducible, then a factor $\epsilon X_0+\sigma$ of $G_*$ equals  one among 
\begin{itemize}
    \item $Y_0Y_1X_0 + \lambda M_*N_*,$
    \item $X_0+\lambda N_* ,$
    \item $Y_1X_0+\lambda M_*.$
\end{itemize}
\end{prop}
\begin{proof}
 It is readily seen that a polynomial $\epsilon(Y_0,Y_1)X_0+\sigma(Y_0,Y_1)$ divides $G_*$ if and only if 
$$\widetilde{G}(Y_0,Y_1):=\epsilon(Y_0,Y_1)^3G_*\left(\frac{\sigma(Y_0,Y_1)}{\epsilon(Y_0,Y_1)},Y_0,Y_1\right)$$ vanishes. We distinguish the following cases:
\begin{itemize}
    \item $Y_1^2X_0+\lambda N_*(Y_0,Y_1)^2.$ \\
    MAGMA computations
    show that the coefficient of $Y_0^5Y_1^7$ in $\widetilde{G}(Y_0,Y_1)$ is  $\lambda A^{1+q^2} (A^qB+1)\neq 0$, so $Y_1^2X_0+\lambda N_* \nmid G_*.$
    \item $Y_0Y_1X_0+ \lambda N_*^2.$ \\
    The coefficient of $Y_0^6Y_1^5$ in $\widetilde{G}(Y_0,Y_1)$ is  $\lambda AB^q(A^qB+1)(A+B^{q+1})^{q^2},$ that is zero if and only if $A+B^{1+q}= 0.$  By replacing $A=B^{q+1}$ in $\widetilde{G}(Y_0,Y_1)$ we get that the coefficient of $Y_0Y_1^7$ in  $\widetilde{G}(Y_0,Y_1)$ is $\lambda^3B^{6(1+q)} (B^{1+q+q^2}+1)^3.$ By replacing $B^{1+q+q^2}=1$ in $\widetilde{G}(Y_0,Y_1),$ MAGMA computations show that the coefficient of $Y_0^7Y_1^7$ in $\widetilde{G}(Y_0,Y_1)$ is  $B^2\neq 0.$ Thus $Y_0Y_1X_0+ \lambda N_*^2 \nmid G_*.$
    \item $Y_0Y_1^2X_0+ \lambda N_*^2.$ \\
    Analogous to the case $\epsilon=Y_0Y_1$ and $\sigma=\lambda N_*^2.$
         
         
     \item $Y_1^2X_0+ \lambda M_*N_*.$ \\ 
     The coefficients of $Y_0^7Y_1^8$ in $\widetilde{G}(Y_0,Y_1)$ reads $\lambda^3 B^3 (A^qB+1)(A^{1+q+q^2}+AB^{q^2}+A^qB+A^{q^2}B^q+B^{1+q+q^2}+1)\neq 0$. So $Y_1^2X_0+ \lambda M_*N_* \nmid G_*.$
     \item $Y_0Y_1^2X_0+ \lambda M_*N_*.$ \\
     Analogous to the case $\epsilon=Y_0Y_1$ and $\sigma=\lambda N_*^2.$
     \item $Y_1^2X_0+ \lambda M_*^2.$ \\
     The coefficient of $Y_0^5Y_1^7$ in  $\widetilde{G}(Y_0,Y_1)$ is $\lambda B^{q+q^2}(A^qB+1)\neq 0$,. Thus $Y_1^2X_0+ \lambda M_*^2 \nmid G_*.$
     \item $Y_0Y_1X_0 +\lambda M_*^2.$ \\
     Analogous to the case $\epsilon=Y_0Y_1$ and $\sigma=\lambda N_*^2.$
     \item $Y_0Y_1^2X_0 +\lambda M_*^2.$ \\
     Analogous to the case $\epsilon=Y_0Y_1$ and $\sigma=\lambda N_*^2.$
     \item $Y_1X_0 + \lambda N_*.$ \\
      The coefficient of $Y_0Y_1^7$ in $\widetilde{G}(Y_0,Y_1)$ reads $ \lambda A^3(A^qB+1)^q\neq 0,$ so $Y_1X_0 + \lambda N_* \nmid G_*.$
     \item $Y_0X_0 + \lambda N_*.$ \\
     The coefficient of $Y_0^7$ in $\widetilde{G}(Y_0,Y_1)$ is $A^{2q}\neq 0$. Thus $Y_0X_0 + \lambda N_* \nmid G_*.$
     \item $X_0+\lambda M_*.$  \\
     The coefficient of $Y_0Y_1^3$ in $\widetilde{G}(Y_0,Y_1)$ reads $\lambda B^{3q}(A^qB+1)^{q^2}\neq 0,$ so $X_0+\lambda M_* \nmid G_*.$
    \item $Y_0X_0+\lambda M_* $   \\ 
    The coefficient of $Y_0^7$ in $\widetilde{G}(Y_0,Y_1)$ is $A^{2q}\neq 0$. Thus $Y_0X_0+\lambda M_* \nmid G_*.$
\end{itemize}
The claim now follows from Proposition \ref{prop casi iniziali da escludere}.
\end{proof}

\begin{rem}
  For $A^qB \in \mathbb{F}_q^* \setminus \left\lbrace 1 \right\rbrace $ and $A^{1+q+q^2}=B^{1+q+q^2},$ it is already known  that  $f(X)=X^{q^2-q+1}+AX^{q^2}+BX$ is a PP over $\mathbb{F}_{q^3}$ for each $q$  (see \cite[Theorem 1]{MR4296215}). In these cases $\mathcal{W}$ decomposes in three components and $G_*=(A^qB+1)^2(Y_0Y_1X_0 + \lambda_1 M_*N_*)(X_0+\lambda_2 N_*)(Y_1X_0+\lambda_3 M_*),$ for suitable  $\lambda_1,\lambda_2,\lambda_3 \in \mathbb{F}_{q^3}.$ 
In what follows we will prove that for $q$ large enough, also the converse is true.  
\end{rem}

\begin{prop} \label{prop due fattori facili}
Let $\alpha \neq 0.$ If $X_0+\lambda N_* \mid G_*$ or $Y_1X_0+\lambda M_* \mid G_*,$ then $A^qB \in \mathbb{F}_q^* \setminus \left\lbrace 1 \right\rbrace $ and $A^{1+q+q^2}=B^{1+q+q^2}.$
\end{prop}
\begin{proof}
 MAGMA computations show that the coefficient of $Y_1^4$ in $\widetilde{G}(Y_0,Y_1)=G_*\left(\lambda N_*,Y_0,Y_1\right)$  is $A^2B^q((A^{q+1}+B^q)\lambda+B^q),$ so $\widetilde{G}(Y_0,Y_1)=0$ yields $(A^{q+1}+B^q)\lambda=B^q, $ and in particular $ A^{q+1}+B^q \neq 0. $
\\ By replacing $\lambda=\frac{B^q}{A^{q+1}+B^q}$ in $\widetilde{G},$ the coefficients of $Y_0^2Y_1^2$ and $Y_0^2Y_1^3$ in $\widetilde{G}$ are $A^{1+q}(A^{q+1}+B^q)(A^{1+2q}+B^{2q+q^2})$ and $A^{1+q}B^q(A^{q+1}+B^q)(A^{1+q+q^2}+A^qB+ A^{q^2}B^q+ B^{1+q+q^2}), $ respectively. Therefore, $\widetilde{G}=0$ implies $A^{1+2q}+B^{2q+q^2}=0=A^{1+q+q^2}+A^qB+ A^{q^2}B^q+ B^{1+q+q^2},$ from which it follows that $B^q(A^{q+1}+B^q)(A^qB+A^{q^2}B^q)=0,$ so $A^qB \in \mathbb{F}_q^* \setminus \left\lbrace 1 \right\rbrace $ and $A^{1+q+q^2}=B^{1+q+q^2}.$ \\
The case $Y_1X_0+\lambda N_*,X_0 \mid G_*$ is analogous.
\end{proof}
\begin{comment}
\begin{rem}
Consider now the pending case, that is $\epsilon X_0+\sigma=Y_0Y_1X_0 + \lambda mn.$ \\
By computing $Res(G_*,Y_0Y_1X_0 + \lambda mn,X_0)$ with MAGMA, we obtain  $Res(G_*,Y_0Y_1X_0 + \lambda mn,X_0)=0 \Rightarrow (A^qB+1)\lambda=1$ or $(A^{1+q+q^2}B  + AB^{1+q^2} + A^qB^2 + A^{q^2}B^{1+q} + B^{2+q+q^2} +B) \lambda + A^{1+q^2} +B=0.$ \\
In the firs case, by replacing $\lambda=(A^qB+1)^{-1}$ in $Res(G_*,Y_0Y_1X_0 + \lambda mn,X_0)=0,$ we get $A^qB \in \mathbb{F}_q$ from which follows that $A^{2+q}=B^{1+2q}$ that is, multiplying by $A^{q^2}B^{q^2},$ $A^{1+q+q^2}=B^{1+q+q^2}.$ \\
The case $\lambda=\frac{A^{1+q^2} +B}{A^{1+q+q^2}B  + AB^{1+q^2} + A^qB^2 + A^{q^2}B^{1+q} + B^{2+q+q^2} +B}$ is more complicated; we can proceed as follows.
\end{rem}
\end{comment}

\begin{prop} \label{prop fattore difficile}
Suppose that $\mathcal{W}$ does not contain an absolutely irreducible component fixed by $\Psi_q$.  If $X_0Y_0Y_1+\lambda M_*N_* \mid G_*,$  then there is another  factor $\epsilon X_0+\sigma$ that divides $G_*.$ In particular  $A^qB \in \mathbb{F}_q^* \setminus \left\lbrace 1 \right\rbrace $ and $A^{1+q+q^2}=B^{1+q+q^2}.$
\begin{comment}
Let $q=2^m,$ with $m \geq 19.$ 
Assume that $f(X)=X^{q^2-q+1}+AX^{q^2}+BX \in \mathbb{F}_{q^3}[X]$ is  a PP over $\mathbb{F}_{q^3}$ and let $\alpha \neq 0.$ If $X_0Y_0Y_1+\lambda mn \mid G_*,$  then there is another  factor $\epsilon X_0+\sigma$ that divides $G_*$ and 
\end{comment}
\end{prop}
\begin{proof}
First, we observe that $Y_0^2 \nmid \alpha(Y_0,Y_1),$ so $X_0Y_0Y_1+\lambda MN$ cannot be a repeated factor of $G_*$  \\
In particular

\begin{equation*}
  \mathcal{W}_1:	\left\{
	\begin{array}{l}
	X_2= \frac{[BX_0Y_1 +N(Y_0,Y_1,Y_2)][A^qX_0Y_2+M(Y_0,Y_1,Y_2)]+X_0^2Y_1Y_2}{Y_1[(A^{q+1}+B^q)X_0Y_2+A\cdot M(Y_0,Y_1,Y_2)]}\\ 
		X_1=\frac{[BX_0Y_1 +N(Y_0,Y_1,Y_2)][A^qX_0Y_2+M(Y_0,Y_1,Y_2)]+X_0^2Y_1Y_2}{Y_2[(A+B^{q+1})X_0Y_1+B^q\cdot N(Y_0,Y_1,Y_2)]}\\
	  X_0=\lambda \cdot \frac{ M(Y_0,Y_1,Y_2)N(Y_0,Y_1,Y_2)}{Y_0Y_1Y_2}
	\end{array}
	\right.  
\end{equation*}

is a non-repeated absolutely irreducible component of $\mathcal{W}$ over $\mathbb{F}_{q^3}.$ 
Thus there are two possible factorization of $G_*:$
\begin{item}
\item [(i)]  $G_*=(X_0Y_0Y_1+\lambda M_*N_*)p(X_0,Y_0,Y_1)$ where $p=\frac{\alpha(Y_0,Y_1)}{Y_0Y_1^2}X_0^2Y_1+ \dots + \frac{1}{\lambda}M_*(Y_0,Y_1)N_*(Y_0,Y_1)$ is absolutely irreducible of degree two in $X_0.$
\item [(ii)] $G_*=(X_0Y_0Y_1+\lambda M_*N_*)(\epsilon_1 X_0+\sigma_1)(\epsilon_2 X_0+\sigma_2)$ with $\epsilon_1 \neq Y_0Y_1 \neq \epsilon_2.$
\end{item}
\\
We want to prove that a factorization of type  $(i)$ cannot exist, namely that $p(Y_0,Y_1)$ can't be absolutely irreducible; from this the claim will follow. \\
Suppose by way of contradiction that $p$ is absolutely irreducible; then 

\begin{equation*}
	\mathcal{W}_2:  \left\{
	\begin{array}{l}
	X_2= \frac{[BX_0Y_1 +N(Y_0,Y_1,Y_2)][A^qX_0Y_2+M(Y_0,Y_1,Y_2)]+X_0^2Y_1Y_2}{Y_1[(A^{q+1}+B^q)X_0Y_2+A\cdot M(Y_0,Y_1,Y_2)]}\\ 
		X_1=\frac{[BX_0Y_1 +N(Y_0,Y_1,Y_2)][A^qX_0Y_2+M(Y_0,Y_1,Y_2)]+X_0^2Y_1Y_2}{Y_2[(A+B^{q+1})X_0Y_1+B^q\cdot N(Y_0,Y_1,Y_2)]}\\
	  \frac{\alpha(Y_0,Y_1)}{Y_0Y_1^2}X_0^2Y_1Y_2+ \dots + \frac{1}{\lambda}M(Y_0,Y_1,Y_2)N(Y_0,Y_1,Y_2)=0
	\end{array}
	\right.  
\end{equation*}

is an absolutely irreducible component of $\mathcal{W}$ defined over $\mathbb{F}_{q^3}$ and $\mathcal{W}=\mathcal{W}_1 \cup \mathcal{W}_2.$ \\
Consider the collineation $\Psi_q$ of $\mathbb{P}^5(\mathbb{F}_{q^3})$ introduced before.  
Since  $\mathcal{V}_1$ is not fixed by $\Psi_q$  
so $\Psi_q(\mathcal{V}_1)=\mathcal{V}_2.$  As $\Psi_q$ is a collineation of $\mathbb{P}^5(\mathbb{F}_{q^3}),$ $\mathcal{V}_1$ and  $\mathcal{V}_2$ have the same degree. Moreover, if $p$ is absolutely irreducible, then it is not a square as a polynomial, so there exists a  subspace  
$
  \pi:  \left\{
	\begin{array}{l}
	Y_0= hY_2\\ 
		Y_1=kY_2
	\end{array}
	\right.  
$
of $\mathbb{P}^5(\mathbb{F}_{q^3})$ of dimension 3 not contained  in the hyperplanes $Y_0=0$ nor $Y_1=0$ (i.e $h\neq 0 \neq k$) such that $p(X_0,h,k)$ has 2 distinct roots, i.e $| \mathcal{V}_2 \cap \pi |=2,$ while $| \mathcal{V}_1 \cap \Tilde{\pi} |=1,$ where $\Psi_q(\Tilde{\pi})=\pi,$  $
  \Tilde{\pi}:  \left\{
	\begin{array}{l}
	Y_2= h^{q^2}Y_1\\ 
		Y_0=k^{q^2}Y_1
	\end{array}
	\right.  
,$ a contradiction. 
So there exists a factor $\epsilon X_0+\sigma$ of $G_*$ different from $X_0Y_0Y_1+\lambda M_*N_*.$
By  Proposition \ref{prop tre fattori lineari} we have that one between $X_0+\lambda N_*$ and $Y_1X_0+\lambda M_*$ divides $G_*;$ the claim now follows from Proposition \ref{prop due fattori facili}.    

\end{proof}
\begin{comment}
\begin{cor}
Let $f(X)=X^{q^2-q+1}+AX^{q^2}+BX \in \mathbb{F}_{q^3}[X]$   a PP and let $\alpha \neq 0.$ If $X_0Y_0Y_1+\lambda mn \mid G_*$ for some $\lambda \in \mathbb{F}_{q^3} ,$ then $A^qB \in \mathbb{F}_q^* \setminus \left\lbrace 1 \right\rbrace $ and $A^{1+q+q^2}=B^{1+q+q^2}.$
\end{cor}
\begin{proof}
By Proposition \ref{prop fattore difficile} and Remark \ref{rem tre fattori lineari} we get that one of $X_0+\lambda n $ and $Y_1X_0+\lambda m$ divides $G_*.$ From Proposition \ref{prop due fattori facili} we deduce the statement.    
\end{proof}
\end{comment}
We can now prove our main result.
\begin{thm}
Let $q=2^m,$ with $m \geq 19.$ Assume that $f(X)=X^{q^2-q+1}+AX^{q^2}+BX \in \mathbb{F}_{q^3}[X]$ is  a PP over $\mathbb{F}_{q^3}$ and let  $A^qB+1 \neq 0 \neq A^{1+q+q^2}+AB^{q^2}+A^qB+A^{q^2}B^q+B^{1+q+q^2}+1.$ Then  $A^{q^2+q+1}=B^{1+q+q^2}$ and $A^qB \in \mathbb{F}_q \setminus \{0,1\}.$  
\end{thm}
\begin{proof}
Assume by  way of contradiction that 
there exists an absolutely irreducible components component $\mathcal{W}_1$ of $\mathcal{W}$ fixed by $\Psi_q.$ 

  Then $\mathcal{W}_1$ corresponds via the projectivity $\Theta$ to an absolutely irreducible $\mathbb{F}_q$-rational component $\mathcal{V}_1$ of $\mathcal{V},$ with  $\dim(\mathcal{V}_1)= \dim(\mathcal{W}_1)\leq \dim(\mathcal{W})\leq 12$ (see Equation \eqref{eq decomp of theta v}).
So, by Theorem \ref{thm lang weil versione tredici terzi} it follows that \begin{equation}
   \#(\mathcal{V}_1 (\mathbb{F}^5_q)) \geq q^2
 -110q^{3/2} - 5\cdot 12^{13/3}q.
\end{equation}
Moreover, $\mathcal{V}_1$ intersects each of the planes 
\begin{equation}
    \Theta(\mathcal{U}_1): \left\{
	\begin{array}{l}
	x_2=0\\ 
		x_1=0\\
	  x_0=0
	\end{array}
	\right. 
	 \ \ \ \ \ \ \ \ \Theta(\mathcal{U}_2):\left\{
	\begin{array}{l}
	 y_2=0\\ 
		y_1=0 \\
	  y_0=0
	\end{array}
	\right.
	\ \ \ \ \ \ \ \ \Theta(\mathcal{U}_3):
	\left\{
	\begin{array}{l}
	 x_2=y_2\\ 
		x_1=y_1 \\
	  x_0=y_0
	\end{array}
	\right.
	\end{equation}
in at most 36 points. Since \[q^2
 -110q^{3/2} - 5\cdot 12^{13/3}q > 36 \] is satisfied for $q \geq 2^{19},$ we deduce the existence of $\mathbb{F}_q$-rational points on $\mathcal{V}$ off $\Theta(\mathcal{U}_1\cup \mathcal{U}_2 \cup \mathcal{U}_3),$ which corresponds to points on $\mathcal{C}_{f_{A,B}}$ off $X^qY^q(X+Y)=0,$ a contradiction to our hypothesis via Proposition \ref{prop collegam curve}.
 
 Thus, no absolutely irreducible components component of $\mathcal{W}$ is fixed by $\Psi_q.$ 
 
 In particular, $\mathcal{W}$ is reducible and so is $G_*$. 
 By Propositions \ref{prop tre fattori lineari},  \ref{prop due fattori facili} and \ref{prop fattore difficile}, $A^{q^2+q+1}=B^{1+q+q^2}$ and $A^qB \in \mathbb{F}_q \setminus \{0,1\}.$

\end{proof}

\section{Case $A^qB=1$}
Set without restriction $B^{1+q+q^2} \neq 1.$   Let us consider again the surface
\begin{equation}\label{Eq:Ultima}
   \mathcal{W}:	\left\{
	\begin{array}{l}
	X_2= \frac{[BX_0Y_1 +N(Y_0,Y_1,Y_2)][A^qX_0Y_2+M(Y_0,Y_1,Y_2)]+X_0^2Y_1Y_2}{Y_1[(A^{q+1}+B^q)X_0Y_2+A\cdot M(Y_0,Y_1,Y_2)]}\\ 
		X_1=\frac{[BX_0Y_1 +N(Y_0,Y_1,Y_2)][A^qX_0Y_2+M(Y_0,Y_1,Y_2)]+X_0^2Y_1Y_2}{Y_2[(A+B^{q+1})X_0Y_1+B^q\cdot N(Y_0,Y_1,Y_2)]}\\
	  G(X_0,Y_0,Y_1,Y_2)=0.
	\end{array}
	\right.  
\end{equation}

Replacing $A=\frac{1}{B^{q^2}},$ $A^q=\frac{1}{B},$ and $A^{q^2}=\frac{1}{B^q}$ in $G$ we get
\begin{equation}  \label{eq G caso aqb eq 1}
 G=(B^{2+q+2q^2})^{-1} \cdot \Tilde{M} \cdot \Tilde{N} \cdot [(B^{1+q+q^2}+1)Y_2X_0+\Tilde{M}][(B^{1+q+q^2}+1)Y_1X_0+B^q\Tilde{N}],    
\end{equation}   where $\Tilde{N}:=B^{1+q^2}Y_0Y_1 +B^{q^2}Y_0Y_2 + Y_1Y_2$ and $\Tilde{M}:=\Psi_q(\Tilde{N})=BY_0Y_1 +Y_0Y_2 + B^{1+q}Y_1Y_2.$

\begin{prop}\label{CASE:pp}
 If $A^qB=1$ and $A^{1+q+q^2}\neq 1,$ then $f_{A,B}$ is a PP over $\mathbb{F}_{q^3}.$  \end{prop}

 \begin{proof}
  Recall that the map  \[\Theta \circ \Lambda:(X,Y) \mapsto  (X,X^q,X^{q^2},Y,Y^{q},Y^{q^2}) = (X_0,X_1,X_2,Y_0,Y_1,Y_2)\]  induces a surjection between the set of affine $\mathbb{F}_{q^3}$-rational points of $\mathcal{C}$ and the set of  points of $\Theta^{-1}(\mathcal{V})$ in $\mathbb{P}^5(\mathbb{F}_{q^3})$ which are fixed by $\Psi_q.$   Consider a point $P=(\overline{X_0},\overline{X_1},\overline{X_2},\overline{Y_0},\overline{Y_1},\overline{Y_2}) \in \Theta^{-1}(\mathcal{V})$ such that   $\Psi_q(P)=P,$ with $\overline{X_i}\overline{Y_j}\neq 0$ and $\overline{X_i}\neq \overline{Y_i},$ for $i,j=0,1,2.$ 
  
  Now, $P \in \mathcal{W}$ and $G(P)=0$ yields one of the following.
  \begin{enumerate}
      \item $\Tilde{N}(P)=0$, which is equivalent to $\Tilde{M}(P)=0$. In this case, $(A^{q+1}+B^q)\overline{X_0}\overline{X_2}\overline{Y_1}\overline{Y_2} =0$ or  $(A+B^{q+1})\overline{X_0}\overline{X_1}\overline{Y_1}\overline{Y_2}=0$, a contradiction to $\overline{X_i}\overline{Y_j}\neq 0$ for $i,j=0,1,2.$
      \item $(B^{1+q+q^2}+1)\overline{Y_2}\overline{X_0}=\Tilde{M}$. From the first equation in \eqref{Eq:Ultima} we get 
      $$[B\overline{X_0}\overline{Y_1} +\Tilde{N}(\overline{Y_0},\overline{Y_1},\overline{Y_2})][A^q\overline{X_0}\overline{Y_2}+\Tilde{M}(\overline{Y_0},\overline{Y_1},\overline{Y_2})]+\overline{X_0}^2\overline{Y_1}\overline{Y_2}=0,$$ which yields, by the second equation in \eqref{Eq:Ultima}, $\overline{X_1}=0$, a contradiction to $\overline{X_i}\overline{Y_j}\neq 0$ for $i,j=0,1,2.$
      \item $(B^{1+q+q^2}+1)\overline{Y_1}\overline{X_0}=B^q\Tilde{N}$. As in the previous case, a contradiction arises. 
  \end{enumerate}
  
  
  Finally, from Equation (\ref{eq caso aqb eq 1 radici}) (Proposition \ref{prop f PP se e solo se norma eq 1}(i)), it follows that $f_{A,B}$ has no trivial root in $\mathbb{F}_{q^3}.$
  The claim follows from Proposition \ref{prop collegam curve}. 
 \end{proof}

\section*{Acknowledgments}
This research was supported by the Italian National Group for Algebraic and Geometric Structures and their Applications (GNSAGA - INdAM).

\bibliographystyle{acm}
\bibliography{biblio.bib}

\end{document}